\documentclass[11pt,twoside]{amsart}
\usepackage{graphicx}
\usepackage{amssymb}
\usepackage{amsmath}
\usepackage{amsthm}
\usepackage{amsfonts}
\usepackage{amsmath, amsfonts, amssymb}
\newtheorem{thm}{Theorem}[section]

\newtheorem{defn}[thm]{Definition}
\newtheorem{rem}[thm]{\bf{Remark}}

\numberwithin{equation}{section}

\begin{document}
\begin{center}\small{In the name of Allah, Most Gracious, Most Merciful.}\end{center}
\vspace{1cm}

\title{Classification of $2$-dimensional evolution algebras, their groups of automorphisms and derivation algebras}
\author{H.Ahmed$^1$, U.Bekbaev$^2$, I.Rakhimov$^3$}

\thanks{{\scriptsize
emails: $^1$houida\_m7@yahoo.com; $^2$bekbaev@iium.edu.my; $^3$rakhimov@upm.edu.my.}}
\maketitle
\begin{center}
\address{$^{1}$Department of Math., Faculty of Science, UPM, Selangor, Malaysia $\&$ Depart. of Math., Faculty of Science, Taiz University, Taiz, Yemen}\\
  \address{$^2$Department of Science in Engineering, Faculty of Engineering, IIUM, Kuala Lumpur, Malaysia}\\
\address{$^3$Department of Math., Faculty of Science $\&$ Institute for Mathematical Research (INSPEM), UPM, Serdang, Selangor, Malaysia}
\end{center}

\begin{abstract}   In the paper we give a complete classification of $2$-dimensional evolution algebras over algebraically closed fields, describe their groups of automorphisms and derivation algebras.
\end{abstract}
{\scriptsize Keywords: evolution algebra, structure constants, automorphism, derivation.

}
\pagestyle{myheadings}
\markboth{\rightline {\sl   H. Ahmed, U.Bekbaev, I.Rakhimov.}}
         {\leftline{\sl   Evolution algebras}}

\bigskip

\section{\bf Introduction}
The classification problem of finite dimensional algebras and description of their invariants with respect to the basis changes is one of the important problems of algebra. One of the interesting class of algebras is the class of evolution algebras. In the present paper we give a complete classification of $2$-dimensional evolution algebras over any algebraically closed field, describe their groups of automorphisms and algebras of derivations. For the further information, related to similar problems, the reader is refereed to \cite{C, T}.
\vskip 0.4 true cm

\section{\bf Classification of $2$-dimensional evolution algebras}
\vskip 0.4 true cm

Let $\mathbb{A}$ be an $n$-dimensional algebra over an algebraic closed field $\mathbb{F}$ with a multiplication $\cdot$ given by a bilinear map $(\mathbf{u},\mathbf{v})\mapsto \mathbf{u}\cdot \mathbf{v}$ whenever $\mathbf{u}, \mathbf{v}\in \mathbb{A}.$ If $e=(e_1,e_2,...,e_n)$ is a
basis of $\mathbb{A}$ over $\mathbb{F}$ then one can represent this bilinear map by a matrix $A=\left(A^k_{i,j}\right) \in Mat(n\times n^2;\mathbb{F})$ as follows \[\mathbf{u}\cdot \mathbf{v}=eA(u\otimes v)\] for any $\mathbf{u}=eu,\mathbf{v}=ev,$
where $u=(u^1, u^2,..., u^n),$ and $ v=(v^1, v^2,..., v^n)$ are column vectors of $\mathbf{u}$ and $\mathbf{v},$ respectively, $e_i\cdot e_j=A^{1}_{i,j}e_1+A^{2}_{i,j}e_2+...+A^{n}_{i,j}e_n$ whenever $i,j=1,2,...,n.$
The matrix $A\in Mat(n\times n^2;\mathbb{F})$ is called the matrix of structure constants (MSC) of $\mathbb{A}$ with respect to the basis $e.$ Further we do not differentiate $\mathbb{A}$ and its MSC $A.$

It is known that under change of the basis $e=(e^1,e^2,...,e^n)$ by $g\in GL(n,\mathbb{F})$ the matrix $A$ changes according to the rule $B=gA(g^{-1})^{\otimes 2}$ that motivates to give the following definition.

\begin{defn} $n$-dimensional algebras $\mathbb{A},\ \mathbb{B},$ given by
	their matrices of structural constants $A,\ B,$ are said to be isomorphic if $B=gA(g^{-1})^{\otimes 2}$ holds true for some $g\in GL(n,\mathbb{F}).$\end{defn}

Let us recall definition of evolution algebras which are the object on focus of the paper.
\begin{defn} An $n$-dimensional algebra $\mathbb{E}$ is said to be an evolution algebra if it admits a basis $\{e^1,e^2,...,e^n\}$ such that $e^i \cdot e^j=0,$ whenever $i\neq j$ and $i,j=1,2,...,n.$\end{defn}
The observations above in $2$-dimensional case is viewed as follows.
Let $\mathbb{\mathbb{A}}$ be a $2$-dimensional algebra then on a basis $e=(e^1,e^2)$ we have
\[A=\left(\begin{array}{cccc}
A^{1}_{1,1} & A^{1}_{1,2}& A^{1}_{2,1}& A^{1}_{2,2} \\
A^{2}_{1,1} & A^{2}_{1,2}& A^{2}_{2,1}& A^{2}_{2,2} \\
\end{array}
\right) \in Mat(2\times 4;\mathbb{F})\]
and $B=gA(g^{-1})^{\otimes 2},$ where for
$g^{-1}=\left(\begin{array}{cccc} \xi_1& \eta_1\\ \xi_2& \eta_2\end{array}\right)\in GL(2,\mathbb{F})$ one has
\[(g^{-1})^{\otimes 2}=g^{-1}\otimes g^{-1}=\left(\begin{array}{cccc} \xi _1^2 & \xi _1 \eta _1 & \xi _1 \eta _1 & \eta _1^2 \\
\xi _1 \xi _2 & \xi _1 \eta _2 & \xi _2 \eta _1 & \eta _1 \eta _2 \\
\xi _1 \xi _2 & \xi _2 \eta _1 & \xi _1 \eta _2 & \eta _1 \eta _2 \\
\xi _2^2 & \xi _2 \eta _2 & \xi _2 \eta _2 & \eta _2^2
\end{array}
\right).\]

Here is a theorem on description of all evolution algebra structures on a $2$-dimensional vector space over $\mathbb{\mathbb{F}}.$
\begin{thm} \label{T1} Over any algebraically closed field $\mathbb{F}$ every nontrivial $2$-dimensional evolution algebra is isomorphic to only one of the algebras listed below by MSC:
\begin{itemize}
	\item $E_{1}(c,b)\simeq E_{1}(b,c)=\left(\begin{array}{cccc} 1 & 0 &0 &b\\ c & 0 & 0 &1\end{array}\right),\ \mbox{ where}\ bc\neq 1,\ (b, c)\in\mathbb{F}^2,$
	\item $E_{2}(b)=\left(\begin{array}{cccc} 1 & 0 &0 &b\\ 1 & 0 & 0 &0\end{array}\right),\ \mbox{ where}\ b\in\mathbb{F},$
	\item $E_{3}=\left(\begin{array}{cccc} 0 & 0 &0 &1\\ 1 & 0 & 0 &0\end{array}\right),$
\item $E_{4}=\left(\begin{array}{cccc} 1 & 0 &0 &1\\ 0 & 0 & 0 &0\end{array}\right),$
	\item $E_5=\left(\begin{array}{cccc} 1 & 0 &0 &-1\\ -1 & 0 & 0 &1\end{array}\right),$
	\item $E_{6}=\left(\begin{array}{cccc} 0 & 0 &0 &1\\ 0 & 0 & 0 &0\end{array}\right).$
\end{itemize}
\end{thm}
\begin{proof} Let $\mathbb{E}$ be a nontrivial evolution algebra given by
	$E=\left(\begin{array}{cccc} a & 0 & 0 &b\\ c & 0 & 0 &d\end{array}\right)$ and \\ $E'=\left(\begin{array}{cccc} \alpha'_1 & \alpha'_2 & \alpha'_3 &\alpha'_4\\ \beta'_1 & \beta'_2 & \beta'_3 &\beta'_4\end{array}\right)=gE(g^{-1})^{\otimes 2},$ where  $g^{-1}=\left(\begin{array}{cccc} \xi_1& \eta_1\\ \xi_2& \eta_2\end{array}\right).$ For the entries of $E'$ we have
\begin{equation} \label{E1}
\begin{array}{ccccccc}
& \alpha'_1&=&\frac{1}{\Delta}(\xi^2_1(a\eta_2-c\eta_1)+\xi^2_2(b\eta_2-d\eta_1)),\\
& \alpha'_2& =& \alpha'_3= \frac{1}{\Delta}(\xi_1\eta_1(a\eta_2-c\eta_1)+\xi_2\eta_2(b\eta_2-d\eta_1)),\\
& \alpha'_4&=&\frac{1}{\Delta}(\eta^2_1(a\eta_2-c\eta_1)+\eta^2_2(b\eta_2-d\eta_1)),\\
& \beta'_1&=&\frac{1}{\Delta}(\xi^2_1(-a\xi_2+c\xi_1)+\xi^2_2(-b\xi_2+d\xi_1)),\\
& \beta'_2& =& \beta'_3=\frac{1}{\Delta}(\xi_1\eta_1(-a\xi_2+c\xi_1)+\xi_2\eta_2(-b\xi_2+d\xi_1)),\\
& \beta'_4&=&\frac{1}{\Delta}(\eta^2_1(-a\xi_2+c\xi_1)+\eta^2_2(-b\xi_2+d\xi_1)),
\end{array}	\end{equation}
where $\Delta=\xi_1\eta_2-\xi_2\eta_1.$
	In particular, one has
	\[\left(\begin{array}{c} \alpha'_2\\ \beta'_2\end{array}\right)=
	\left(\begin{array}{cc} \xi_1 & \eta_1\\ \xi_2 & \eta_2\end{array}\right)^{-1}\left(\begin{array}{cc} a & b\\ c & d\end{array}\right)\left(\begin{array}{c} \xi_1\eta_1\\ \xi_2\eta_2\end{array}\right).\]  Note also that \[\left(\begin{array}{cc} \alpha'_1 & \alpha'_4\\ \beta'_1 & \beta'_4\end{array}\right)=\left(\begin{array}{cc} \xi_1 & \eta_1\\ \xi_2 & \eta_2\end{array}\right)^{-1}\left(\begin{array}{cc} a & b\\ c & d\end{array}\right)\left(\begin{array}{cc}  \xi^2_1 & \eta^2_1\\ \xi^2_2 & \eta^2_2\end{array}\right),\] which shows that $\alpha'_1\beta'_4-\alpha'_4\beta'_1=0$ whenever $ad-bc=0.$
	
	Now we are searching possibilities to choose the base changes that make $\alpha'_1,\ \alpha'_4,\ \beta'_1,\ \beta'_4$ as simple as possible and
	\[\alpha'_2 = \alpha'_3=\beta'_2 = \beta'_3=0\]
that means \begin{equation}\label{E2}
             \left(\begin{array}{cc} \xi_1 & \eta_1\\ \xi_2 & \eta_2\end{array}\right)^{-1}\left(\begin{array}{cc} a & b\\ c & d\end{array}\right)\left(\begin{array}{c} \xi_1\eta_1\\ \xi_2\eta_2\end{array}\right)=\left(\begin{array}{c} 0 \\ 0 \end{array}\right).
           \end{equation}

We make use the following case by case considerations.\\
	 \underline{Case 1. $ad-bc\neq 0.$} In this case (\ref{E2}) is equivalent to $\xi_1\eta_1= \xi_2\eta_2=0.$ Let us consider $g=\left(\begin{array}{cc} \xi_1 & 0\\ 0 & \eta_2\end{array}\right).$ Then $\Delta=\xi_1\eta_2$ and \[ \alpha'_1=a\xi_1,\
	\alpha'_4=b\frac{\eta^2_2}{\xi_1},\
	\beta'_1=c\frac{\xi^2_1}{\eta_2},\
	\beta'_4=d\eta_2.\]
	
	Due to $ad-bc\neq 0$ one has the following cases:
	
	\underline{Case 1.1. $a\neq 0,\ d\neq 0.$} In this case one can make $ \alpha'_1=1,  \beta'_4=1$ to get \[E_{1}(b,c)=\left(\begin{array}{cccc} 1 & 0 &0 &b\\ c & 0 & 0 &1\end{array}\right), \mbox{ where } bc\neq 1.\]
	
	\underline{Case 1.2. $a\neq 0,\ d= 0.$} In this case $\beta'_4=0 $ and one can make $ \alpha'_1=1,  \beta'_1=1$ to get \[E_{2}(b)=\left(\begin{array}{cccc} 1 & 0 &0 &b\\ 1 & 0 & 0 &0\end{array}\right), \mbox{ where } b\neq 0.\]
	
	\underline{Case 1.3. $a=0,\ d\neq 0.$} In this case $ \alpha'_1=0$ and one can make $\beta'_4=1,  \alpha'_4=1$ to get $E'=\left(\begin{array}{cccc} 0 & 0 &0 &1\\ c & 0 & 0 &1\end{array}\right),$ where $c\neq 0.$ But $E'$ is isomorphic to $E_{2}(c).$
	
	\underline{Case 1.4 $a=0,\ d= 0.$} In this $ \alpha'_1=0,  \beta'_4=0$ and one can make $\beta'_1=\alpha'_4=1$ to get \[E_{3}=\left(\begin{array}{cccc} 0 & 0 &0 &1\\ 1 & 0 & 0 &0\end{array}\right).\]\\	
\underline{Case 2. $ad-bc= 0.$}
	
	\underline{Case 2.1. Both $(a,b),\ (c,d)$ are nonzero and $(c,d)=\lambda (a,b).$}  In this case
	(\ref{E1}) is equivalent to \[a\xi_1\eta_1+b\xi_2\eta_2=0,\
	\alpha'_1=\frac{\eta_2-\lambda\eta_1}{\Delta}(a\xi^2_1+b\xi^2_2),\
	\alpha'_4=\frac{\eta_2-\lambda\eta_1}{\Delta}(a\eta^2_1+b\eta^2_2),\]
	\[\beta'_1=-\frac{\xi_2-\lambda\xi_1}{\Delta}(a\xi^2_1+b\xi^2_2),\
	\beta'_4=-\frac{\xi_2-\lambda\xi_1}{\Delta}(a\eta^2_1+b\eta^2_2).\]
	
	\quad \underline{Case 2.1.1. $a+b\lambda^2\neq 0.$} Put $\xi_2-\lambda\xi_1=0.$ Then $\xi_1\neq 0,$ the equation $a\xi_1\eta_1+b\xi_2\eta_2=\xi_1(a\eta_1+b\lambda\eta_2)$ implies $a\eta_1+b\lambda\eta_2=0.$ 	
	\underline{If $b\neq 0$} then $\frac{\eta_2}{\eta_1}=-\frac{a}{b\lambda},\ \Delta=\xi_1(\eta_2-\lambda\eta_1)$ and \[\beta'_1=\beta'_4=0,  \alpha'_1=(a+b\lambda^2)\xi_1, \alpha'_4=\frac{\eta^2_1}{\xi_1}\frac{a(a+b\lambda^2)}{b\lambda^2}.\]
	It implies that in this case one can make $\alpha'_1=1,\ \alpha'_4$ equal to one or zero, depending on $a,$  to get
	\[E_4=\left(\begin{array}{cccc} 1 & 0 &0 &1\\ 0 & 0 & 0 &0\end{array}\right)\ \mbox{or}\ E'=\left(\begin{array}{cccc} 1 & 0 &0 &0\\ 0 & 0 & 0 &0\end{array}\right).\]
	The last $E'$ is isomorphic to $E_{2}(0).$ 	
	\underline{If $b=0$} then $\eta_1$ has to be zero, $\alpha'_1=a\xi_1,\
	\alpha'_4=0,$ so by making $\alpha'_1=1$ one gets $E'=\left(\begin{array}{cccc} 1 & 0 &0 &0\\ 0 & 0 & 0 &0\end{array}\right),$ which is isomorphic to $E_{2}(0).$
	
\quad	\underline{Case 2.1.2. $a+b\lambda^2= 0.$} Note that in this case $a,b, \lambda$ have to be nonzero and therefore one can make $\xi_2=\eta_1=0.$ Then $\Delta=\xi_1\eta_2,$
	and \[   \alpha'_1=a\xi_1,
	\alpha'_4=\frac{b\eta^2_2}{\xi_1},
	\beta'_1=\frac{a\lambda\xi^2_1}{\eta_2}, \beta'_4=b\lambda\eta_2.\]
	It implies that one can make $\alpha'_1=1, \beta'_4=1$ to get $\alpha'_4=\frac{a}{b\lambda^2}=-1,\ \beta'_1=\frac{b\lambda^2}{a}=-1$ and
	\[E_5=\left(\begin{array}{cccc} 1 & 0 &0 &-1\\ -1 & 0 & 0 &1\end{array}\right).\]
	

	\underline{Case 2.2. $c=d= 0.$} In this case\\
	\[ \alpha'_1=\frac{\eta_2}{\Delta}(a\xi^2_1+b\xi^2_2),\ \
	\alpha'_2 = \alpha'_3= \frac{\eta_2}{\Delta}(a\xi_1\eta_1+b\xi_2\eta_2),\ \
	\alpha'_4=\frac{\eta_2}{\Delta}(a\eta^2_1+b\eta^2_2),\]
	\[\beta'_1=-\frac{\xi_2}{\Delta}(a\xi^2_1+b\xi^2_2),\ \
	\beta'_2 = \beta'_3=-\frac{\xi_2}{\Delta}(a\xi_1\eta_1+b\xi_2\eta_2),\ \ \beta'_4=-\frac{\xi_2}{\Delta}(a\eta^2_1+b\eta^2_2).\]
	
	Taking $\xi_2=0,\ \eta_1=0$ results in
	\[ \alpha'_1=a\xi_1,\ \alpha'_2=\alpha'_3=0,\ \alpha'_4=\frac{b\eta^2_2}{\xi_1},\
	\beta'_1=\beta'_2=\beta'_3=\beta'_4=0.\]
	
	\quad \underline{Case 2.2.1. $a\neq 0.$} Then one can make $\alpha'_1=1,\ \alpha'_4 =1$ or $0,$ depending on $b$ to get
	\[E_{4}=\left(\begin{array}{cccc} 1 & 0 &0 &1\\ 0 & 0 & 0 &0\end{array}\right)\ \mbox{or}\
	E'=\left(\begin{array}{cccc} 1 & 0 &0 &0\\ 0 & 0 & 0 &0\end{array}\right),\]  respectively. The last $E'$ is isomorphic to $E_{2}(0).$
	
\quad	\underline{Case 2.2.2. $a= 0.$} Then
	\[ \alpha'_1=0,\ \alpha'_2=\alpha'_3=0,\ \alpha'_4=\frac{b\eta^2_2}{\xi_1},\
	\beta'_1=\beta'_2=\beta'_3=\beta'_4=0,\] and one can make $\alpha'_4=1$ to get
	\[E_{6}=\left(\begin{array}{cccc} 0 & 0 &0 &1\\ 0 & 0 & 0 &0\end{array}\right).\]
	
	\underline{Case 2.3. $a=b= 0.$} In this case we have
	\[ \alpha'_1=-\frac{\eta_1}{\Delta}(c\xi^2_1+d\xi^2_2),\ \
	\alpha'_2 = \alpha'_3= -\frac{\eta_1}{\Delta}(c\xi_1\eta_1+d\xi_2\eta_2),\ \
	\alpha'_4=-\frac{\eta_1}{\Delta}(c\eta^2_1+d\eta^2_2),\]
	\[\beta'_1=\frac{\xi_1}{\Delta}(c\xi^2_1+d\xi^2_2),\ \
	\beta'_2 = \beta'_3=\frac{\xi_1}{\Delta}(c\xi_1\eta_1+d\xi_2\eta_2),\ \ \beta'_4=\frac{\xi_1}{\Delta}(c\eta^2_1+d\eta^2_2),\] which is similar to that of $c=d=0$ case. A justification, similar to the case of $c=d=0$ shows that such algebras are isomorphic to those considered earlier.\end{proof}
\begin{rem}  The following classification theorem on complex evolution algebras has been stated in \cite{C}. By the next theorem we restated the result by MSC.\end{rem}
	\begin{thm} \label{T2} Every nontrivial $2$-dimensional complex evolution algebra is isomorphic
		to exactly one evolution algebra presented below by its MSC:
		\[E_1: \left(\begin{array}{cccc} 1 & 0&0&0\\ 0 &0&0&0\end{array}\right),\ E_2: \left(\begin{array}{cccc} 1 & 0&0&1\\ 0 &0&0&0\end{array}\right), \]
		\[E_3: \left(\begin{array}{cccc} 1 & 0&0&-1\\ 1 &0&0&-1\end{array}\right),\ E_4: \left(\begin{array}{cccc} 0 & 0&0&0\\ 1 &0&0&0\end{array}\right), \]
		\[E_{5_{a,b}}: \left(\begin{array}{cccc} 1 & 0&0&b\\ a &0&0&1\end{array}\right),\ E_{6_c}: \left(\begin{array}{cccc} 0 & 0&0&1\\ 1 &0&0&c\end{array}\right), \]
		where $ab\neq 1,\ c\neq 0$ and $E_{5_{a,b}}\simeq E_{5_{b,a}},\ E_{6_c}\simeq E_{6_{c'}}\Leftrightarrow
		\frac{c}{c'}=cos\frac{2k\pi}{3}+isin\frac{2k\pi}{3}$ for some $k\in\{0,1,2\}.$\end{thm}
	
	Let us compare the list given above and that of Theorem \ref{T1}: \[E_2(0)\simeq E_1,\ E_4\simeq E_2,\ E_5\simeq E_3,\ E_6\simeq E_4,\ E_1(a,b)= E_{5_{a,b}},\ E_2(c^{-3})\simeq E_{6_c},\] where the last isomorphism is due to $\left(\begin{array}{cccc} 1 & 0&0&c^{-3}\\ 1 &0&0&0\end{array}\right)=gE_{6_c}(g^{-1})^{\otimes 2}$ at $g=\left(\begin{array}{cc}0&c\\c^2&0\end{array}\right).$ So we conclude that in Theorem \ref{T2} the algebra $E_3=E_{6_0}$ is missed.

\section{\bf The groups of automorphisms of $2$-dimensional evolution algebras}
Let $i\in \mathbb{F}$ stand for an element with $i^2=-1,\ I=\left(
\begin{array}{cc}
1 & 0 \\
0 & 1
\end{array}
\right)$ and $g=\left(
\begin{array}{cc}
	x & y \\
	z & t
\end{array}
\right).$

If $\mathbb{E}$ is an algebra given by MSC $E$ then its group of automorphisms  $Aut(E)$ is presented as follows
\begin{equation}\label{aut}
  Aut(E)=\{g\in GL(2,\mathbb{F}): \  gE-E(g\otimes g)=0\}.
\end{equation}

\begin{thm} Automorphism groups of all evolution algebra structures on $2$-dimensional vector space over an algebraically closed field  $\mathbb{F}$ of characteristic not $2$ are given as follows.
\begin{itemize}
\item $Aut(E_1(b,c))=\{I\},\ \mbox{if}\ b\neq c,$

\item $Aut(E_1(b,b))=\left\{I,\left(
	\begin{array}{cc}
	0 & 1 \\
	1 & 0 \\
	\end{array}
	\right)\right\},\ \mbox{if}\ b^2\neq 1,$
\item $Aut(E_2(b))=\{I\},\ \mbox{if}\ b\neq 0,$
\item $Aut(E_2(0))=\left\{\left(
\begin{array}{cc}
1 & 0 \\
t & 1-t \\
\end{array}
\right):\ t\neq 1 \right\},$
\item $Aut(E_3)=\left\{I, \left(
\begin{array}{cc}
0 & 1 \\
1 & 0 \\
\end{array}
\right),\left(
\begin{array}{cc}
t & 0 \\
0 & t^2 \\
\end{array}
\right) \left(
\begin{array}{cc}
t^2 & 0 \\
0 & t \\
\end{array}
\right),  \left(
\begin{array}{cc}
0 & t \\
t^2 & 0 \\
\end{array}
\right) ,  \left(
\begin{array}{cc}
0 & t^2 \\
t & 0 \\
\end{array}
\right) \right\},$\\

\hfill where $t=-\frac{1}{2}+i\frac{\sqrt{3}}{2},$
\item $Aut(E_4)=\left\{I, \left(
\begin{array}{cc}
1 & 0 \\
0 & -1 \\
\end{array}
\right) \right\},$
\item $Aut(E_5))=\left\{ \left(\begin{array}{cc}t&1-t\\ 1-t&t\end{array}\right):\ \ t\neq \frac{1}{2}\right\},$
\item $Aut(E_6)=\left\{\left(
\begin{array}{cc}
t^2 & s \\
0 & t \\
\end{array}
\right) :\ \ t\neq 0,\ s\in \mathbb{F} \right\}.$
\end{itemize}
\end{thm}
\begin{proof}
Indeed, let $E=E_1(b,c)=\left(
\begin{array}{cccc}
1 & 0 & 0 & b \\
c & 0 & 0 & 1
\end{array}
\right).$ Then
\[gE_1(b,c)-E_1(b,c)(g\otimes g)=\left(
\begin{array}{cccc}
x-x^2+c y-b z^2 & -x y-b t z & -x y-b t z & -b t^2+b x+y-y^2 \\
c t-c x^2+z-z^2 & -c x y-t z & -c x y-t z & t-t^2-c y^2+b z
\end{array}
\right).\] Therefore to describe the automorphisms we have to solve the system of equations:
\begin{equation} \label{SE1}
\begin{array}{ccccccc}
 & \ x-x^2+c y-b z^2=0 ,\\ & \ ct-cx^2+z-z^2=0,\\ &\ \quad \quad \quad -xy-b t z=0,\\ & \ \quad \quad \quad -cxy-tz=0,\\  & -bt^2+b x+y-y^2=0,\\ & \quad t-t^2-c y^2+b z=0.
\end{array}
\end{equation}
The equations $3$ and $4$ of the system of equations (\ref{SE1}) imply that $tz(bc-1)=0.$\\
\underline{Case 1. $b\neq c.$} In this case, the system above has only one solution $g=I$ due to $bc-1\neq0.$\\
\underline{Case 2. $b=c.$} In this case also the system of equations (\ref{SE1}) has only one solution: $\left(
\begin{array}{cc}
	0 & 1 \\
	1 & 0 \\
\end{array}
\right).$

Let now $E=E_2(b)=\left(
\begin{array}{cccc}
1 & 0 & 0 & b \\
1 & 0 & 0 & 0
\end{array}
\right).$ Then
\[gE_2(b)-E_2(b)(g\otimes g)=\left(
\begin{array}{cccc}
x-x^2+y-b z^2 & -xy-btz & -xy-btz & -b t^2+b x-y^2 \\
t-x^2+z & -x y & -x y & -y^2+b z
\end{array}\right).\] To find $g$ one has to solve the following system of equations with respect to $x, y, z$ and $t$:
\begin{equation} \label{SE2}
\begin{array}{cccccc}
 & \ x-x^2+y-b z^2=0,\\  & \ \quad \quad \quad t-x^2+z=0, \\ & \ \quad \quad \quad -x y-b t z=0,\\ & \ \quad \quad \quad \quad \quad \quad -x y=0,\\ & \ \ \quad -b t^2+b x-y^2=0,\\  & \ \quad \quad \quad \quad -y^2+bz=0.\end{array}\end{equation}
We make the following case by case consideration:\\
\underline{Case 1. $b\neq 0.$} Due to $xy=zt=0$ one has only two cases:

\underline{Case 1.1. $x=t=0,\ yz\neq 0.$} In this case the equation $2$ of the system of equations (\ref{SE2}) implies $z=0.$ So there is no nontrivial $g$, $Aut(E_2(b))=\{I\}.$

\underline{Case 1.2. $xt\neq 0,\ y=z=0.$} In this case we have $x=t=1$, hence, $g=I.$\\
\underline{Case 2. $b=0.$} One has $y=x^2-x,\ t=x^2-z,\ y=0$ therefore $x$ has to be $1,\ t=1-z$ and
$g=\left(
\begin{array}{cc}
1 & 0 \\
z & 1-z \\
\end{array}
\right),$ where $z\neq 1.$

Let $E=E_3=\left(
\begin{array}{cccc}
0 & 0 & 0 & 1 \\
1 & 0 & 0 & 0
\end{array}
\right).$ Then
\[gE_3-E_3(g\otimes g)= \left(
\begin{array}{cccc}
y-z^2 & -t z & -t z & -t^2+x \\
t-x^2 & -x y & -x y & -y^2+z
\end{array}
\right).\] Due to (\ref{aut})  we have two cases:\\
\underline{Case 1. $xt\neq 0,\ y=z=0.$} In this case due to $t=x^2,\ x=t^2$ one has $x=1,\ t=1$ or $x=-\frac{1}{2}+i\frac{\sqrt{3}}{2},\
t=-\frac{1}{2}-i\frac{\sqrt{3}}{2}$ or $x=-\frac{1}{2}-i\frac{\sqrt{3}}{2},\
t=-\frac{1}{2}+i\frac{\sqrt{3}}{2}.$\\
\underline{Case 2. $x=t= 0,\ yz\neq 0.$} Similarly in this case one comes to $y=1,\ z=1$ or $y=-\frac{1}{2}+i\frac{\sqrt{3}}{2},\
z=-\frac{1}{2}-i\frac{\sqrt{3}}{2}$ or $y=-\frac{1}{2}-i\frac{\sqrt{3}}{2},\
z=-\frac{1}{2}+i\frac{\sqrt{3}}{2}.$

Let $E=E_4=\left(
\begin{array}{cccc}
1 & 0 & 0 & 1 \\
0 & 0 & 0 & 0
\end{array}
\right).$ Then
 \[gE_4-E_4(g\otimes g)=\left(
\begin{array}{cccc}
x-x^2-z^2 & -x y-t z & -x y-t z & -t^2+x-y^2 \\
z & 0 & 0 & z
\end{array}
\right).\]
In this case $g=I$ or $ g=\left(
\begin{array}{cc}
1 & 0 \\
0 & -1 \\
\end{array}
\right).$

Let $E=E_5=\left(
\begin{array}{cccc}
1 & 0 & 0 & -1 \\
-1 & 0 & 0 & 1
\end{array}
\right).$ Then
\[gE_5-E_5(g\otimes g)=
\left(
\begin{array}{cccc}
 x-x^2-y+z^2 & -x y+t z & -x y+t z & t^2-x+y-y^2 \\
 -t+x^2+z-z^2 & x y-t z & x y-t z & t-t^2+y^2-z
\end{array}
\right).\]
Due to (\ref{aut}) one has the system of equations:
\begin{equation} \label{SE3}
\begin{array}{ccccccc}
& \ x-x^2-y+z^2=0, \\
  & -t+x^2+z-z^2=0,\\
  & \ \quad \quad \quad \quad x y-t z=0, \\
 & \ \quad \quad \quad \quad xy-zt=0, \\
   & \ \quad t^2-x+y-y^2 =0, \\
 & \ \quad  t-t^2+y^2-z=0
\end{array}
\end{equation}
 which can be rewritten as follows
$$y=x-x^2+z^2,\ \ t=x^2+z-z^2,$$
$$x(x^2-x-z^2)+z(x^2+z-z^2)=0,$$ $$x^2-z^2=-(x^2-x-z^2)^2+(x^2-(z^2-z))^2.$$
\underline{Case 1. $z\neq 0.$} Then $x^2+z-z^2=\frac{-x(x^2-x-z^2)}{z}$ and substitution it into the last equation of the system of equations (\ref{SE3})
implies that $$z^2(x^2-z^2)=(x^2-z^2)(x^2-x-z^2)^2,$$ $$(x^2-z^2)((x^2-x-z^2)^2-z^2)=0.$$

\underline{Case 1.1. $x^2-z^2=0.$} Then $x=\pm z,\ y=\pm z,\ t=z$, i.e., $g$ is singular.

\underline{Case 1.2.  $(x^2-x-z^2)^2-z^2=0.$} Then one has $x^2-x-z^2=\pm z,\ y=\mp z,\ t=x\pm z+z,\ x\pm z+z=\frac{-x(\pm z)}{z}=\mp x.$ Therefore there are two cases:

\quad \underline{Case 1.2.1. $x^2-x-z^2=z,\ y=-z,\ t=x+2z,\ 2x+2z=0.$} One has  $x=-z,\ y=-z,\ t=z$ and $g$ is singular.

\quad \underline{Case 1.2.2. $x^2-x-z^2=- z,\ y=z,\ t=x.$} This case implies that $z=1-x$ and $g= \left(\begin{array}{cc}x&1-x\\ 1-x&x\end{array}\right)$ is an automorphism, where $x\neq\frac{1}{2}.$\\
\underline{Case 2. $z= 0.$} Then one has $y=-(x^2-x),\ t=x^2,\ x^2(x-1)=0$ and $x^2=-(x^2-x)^2+x^4.$ So $x=1,\ y=0,\ t=1$ and one gets the trivial automorphism.

Let $E=E_6=\left(
\begin{array}{cccc}
0 & 0 & 0 & 1 \\
0 & 0 & 0 & 0
\end{array}
\right).$ Then
\[gE_6-E_6(g\otimes g)=\left(
\begin{array}{cccc}
-z^2 & -t z & -t z & -t^2+x \\
0 & 0 & 0 & z
\end{array}
\right),\] so
$g=\left(
\begin{array}{cc}
t^2 & y \\
0 & t \\
\end{array}
\right),$ where $t\neq 0.$\end{proof}
In the case of characteristic $2$ the corresponding result is as follows.
\begin{thm} Automorphism groups of all $2$-dimensional evolution algebras over an algebraically closed field  $\mathbb{F}$ of characteristic $2$ are given as follows.
\begin{itemize}
\item $Aut(E_1(b,c))=\{I\},$ if $ b\neq c,$
\item $Aut(E_1(b,b))=\left\{I,\left(
	\begin{array}{cc}
	0 & 1 \\
	1 & 0 \\
	\end{array}
	\right)\right\},$ if $b^2\neq 1,$
	\item $Aut(E_2(b))=\{I\},$ if $b\neq 0,$
	\item $Aut(E_2(0))=\left\{\left(
	\begin{array}{cc}
	1 & 0 \\
	t & 1-t \\
	\end{array}
	\right):\ t\neq 1 \right\},$
	\item $Aut(E_3)=\left\{I, \left(
	\begin{array}{cc}
	0 & 1 \\
	1 & 0 \\
	\end{array}
	\right),\left(
	\begin{array}{cc}
	t & 0 \\
	0 & t^2 \\
	\end{array}
	\right) \left(
	\begin{array}{cc}
	t^2 & 0 \\
	0 & t \\
	\end{array}
	\right),  \left(
	\begin{array}{cc}
	0 & t \\
	t^2 & 0 \\
	\end{array}
	\right) ,  \left(
	\begin{array}{cc}
	0 & t^2 \\
	t & 0 \\
	\end{array}
	\right) \right\},$\\

\hfill where $t^2+t+1=0,$
	\item $Aut(E_4)=\{I\},$
	\item $Aut(E_5)=\left\{ \left(\begin{array}{cc}t&1-t\\ 1-t&t\end{array}\right):\ \ t\in \mathbb{F} \right\},$
	\item $Aut(E_6)=\left\{\left(
	\begin{array}{cc}
	t^2 & s \\
	0 & t \\
	\end{array}
	\right) :\ \ t\neq 0,\ s\in \mathbb{F} \right\}.$
\end{itemize}
\end{thm}

\section{\bf Derivation algebras of $2$-dimensional evolution algebras }
If $\mathbb{E}$ is an algebra given by MSC $E$ then the algebra of its derivations $Der(E)$ is presented as follows
\[ Der(E)=\{D\in M(2;\mathbb{F}):\  E(D\otimes I+I\otimes D)-DE=0\}.\]

\begin{thm} Derivations of all $2$-dimensional evolution algebras over an algebraically closed field $\mathbb{F}$ of characteristic not $2,\ 3$ are given as follows.
\begin{itemize}
\item $Der(E_1(b,c))=\{0\},$
\item $Der(E_2(b))=\{0\},$ if $b\neq 0,$
\item $Der(E_2(0))=\left\{\left(
\begin{array}{cc}
0 & 0 \\
t & -t \\
\end{array} \right):\ t\in \mathbb{F}\right\} ,$
\item $Der(E_3)=Der(E_4)=\{0\},$
\item $Der(E_5)=\left\{\left(
\begin{array}{cc}
-t & t \\
t & -t \\
\end{array} \right):\ t\in \mathbb{F}\right\},$
\item $Der(E_6)=\left\{\left(
\begin{array}{cc}
2t & s \\
0 & t \\
\end{array} \right):\ t,s\in \mathbb{F}\right\}.$
\end{itemize}
\end{thm}
\begin{proof}
Let $D=\left(\begin{array}{cc}
x & y \\
z & t \\
\end{array}\right)$ be any element in $M(2;\mathbb{F}).$\\
If $E=E_1(b,c)=\left(
\begin{array}{cccc}
1 & 0 & 0 & b \\
c & 0 & 0 & 1
\end{array}
\right)$ then\\
\[ E_1(b,c)(D\otimes I+I\otimes D)-DE_1(b,c)=\left(
\begin{array}{cccc}
x-c y & y+b z & y+b z & 2 b t-b x-y \\
-c t+2 c x-z & c y+z & c y+z & t-b z
\end{array}
\right)\] and one has to solve the system of equations:
\begin{equation} \label{SE4}
\begin{array}{ccccccc} & \ \quad \quad \quad x - c y=0,\\ & \ -c t + 2 c x - z=0,\\ & \ \quad \quad \quad y + b z=0,\\ & \ \quad \quad \quad c y + z=0,\\ & \ \ 2 b t-b x-y=0,\\ & \ \quad \quad \quad t-b z=0\end{array}\end{equation} to find the derivations. The equations $3,\ 4$ of the system of equations (\ref{SE4}) imply $ z(1-bc)=0.$ Therefore due to $bc\neq 1$ one has $x=y=t=z=0$ and $D=0,$ which implies that  $Der(E_1(b,c))=\{0\}.$

Let $E=E_2(b)=\left(
\begin{array}{cccc}
1 & 0 & 0 & b \\
1 & 0 & 0 & 0
\end{array}
\right).$ Then
\[E_2(b)(D\otimes I+I\otimes D)-DE_2(b)=\left(
\begin{array}{cccc}
x-y & y+b z & y+b z & 2 b t-b x \\
-t+2 x-z & y & y & -b z
\end{array}
\right),\] which implies due to (\ref{aut}) that $x=y=0,\ t=-z,\ bz=0.$ So $E_2(b)$ has a nontrivial derivation $D=
\left(
\begin{array}{cc}
0 & 0 \\
z & -z \\
\end{array} \right)$ if and only if $b=0.$

Let $E=E_3=\left(
\begin{array}{cccc}
0 & 0 & 0 & 1 \\
1 & 0 & 0 & 0
\end{array}
\right).$ Then
\[ E_3(D\otimes I+I\otimes D)-DE_3= \left(
\begin{array}{cccc}
-y & z & z & 2 t-x \\
-t+2 x & y & y & -z
\end{array}
\right)\] and one gets $D=0.$

Let $E=E_4=\left(
\begin{array}{cccc}
1 & 0 & 0 & 1 \\
0 & 0 & 0 & 0
\end{array}
\right).$ Then
\[E_4(D\otimes I+I\otimes D)-DE_4=\left(
\begin{array}{cccc}
x & y+z & y+z & 2 t-x \\
-z & 0 & 0 & -z
\end{array}
\right)\] and we get $D=0.$

Let $E=E_5=\left(
\begin{array}{cccc}
1 & 0 & 0 & -1 \\
-1 & 0 & 0 & 1
\end{array}
\right).$ Then
\[ E_5(D\otimes I+I\otimes D)-DE_5=\left(
\begin{array}{cccc}
x+y & y-z & y-z & -2t+x-y \\
t-2x-z & -y+z & -y+z & t+z
\end{array}
\right)\] and one easily comes to $D=\left(
\begin{array}{cc}
-z & z \\
z & -z \\
\end{array}
\right).$

Let $E=E_6=\left(
\begin{array}{cccc}
0 & 0 & 0 & 1 \\
0 & 0 & 0 & 0
\end{array}
\right).$ Then
\[E_6(D\otimes I+I\otimes D)-DE_6=\left(
\begin{array}{cccc}
0 & z & z & 2 t-x \\
0 & 0 & 0 & -z
\end{array}
\right)\] and one obtains
$D=\left(
\begin{array}{cc}
2t & y \\
0 & t \\
\end{array}
\right).$\end{proof}

 Here are the corresponding results in the case of characteristic $2$ and $3.$

\begin{thm} Derivations of all $2$-dimensional evolution algebras over an algebraically closed field $\mathbb{F}$ of characteristic $2$ are given as follows.
\begin{itemize}
	\item $Der((E_1(b,c))=\{0\}\ ,$ 	
	\item $Der(E_2(b))=\{0\},$ if $b\neq 0,$
	\item $Der(E_2(0))=\left\{\left(
	\begin{array}{cc}
	0 & 0 \\
	t & -t \\
	\end{array} \right):\ t\in \mathbb{F}\right\},$
 \item $Der(E_3)=\{0\},$
 \item $Der(E_4)=\left\{\left(
	\begin{array}{cc}
	0 & 0 \\
	0 & t \\
	\end{array} \right):\ t\in \mathbb{F}\right\},$
	\item $Der(E_5)=\left\{\left(
	\begin{array}{cc}
	t & -t \\
	t & -t \\
	\end{array} \right):\ t\in \mathbb{F}\right\},$
	\item $Der(E_6)=\left\{\left(
	\begin{array}{cc}
	0 & s \\
	0 & t \\
	\end{array} \right):\ t,\ s\in \mathbb{F}\right\}.$
\end{itemize}
\end{thm}
\begin{thm} Derivations of all $2$-dimensional evolution algebras over an algebraically closed field $\mathbb{F}$ of characteristic $3$ are given as follows.
\begin{itemize}
\item $Der(E_1(b,c))=\{0\} ,$
\item $Der(E_2(b))=\{0\},$ if $b\neq 0,$
\item $Der(E_2(0))=\left\{\left(
	\begin{array}{cc}
	0 & 0 \\
	t & -t \\
	\end{array}\right):\ t\in \mathbb{F}\right\},$
\item $Der(E_3)=\left\{\left(
	\begin{array}{cc}
	2t & 0 \\
	0 & t \\
	\end{array}\right):\ t\in \mathbb{F}\right\},$
\item $Der(E_4)=\{0\},$
\item $Der(E_5)=\left\{\left(
	\begin{array}{cc}
	-t & t \\
	t & -t \\
	\end{array}\right):\ t\in \mathbb{F}\right\},$
\item $Der(E_6)=\left\{\left(
	\begin{array}{cc}
	2t & s \\
	0 & t \\
	\end{array}\right):\ t,\ s\in \mathbb{F}\right\}.$
\end{itemize}\end{thm}

\begin{center}{\textbf{Acknowledgments}}
\end{center}
The second author's research is supported by FRGS14-153-0394, MOHE and the third author acknowledges MOHE for a support by grant 01-02-14-1591FR.
\vskip 0.4 true cm

\end{document}